\documentclass{icm2010}

\usepackage{latexsym,amsmath,amssymb}
\usepackage[latin1]{inputenc}
\usepackage[english]{babel}
\usepackage{indentfirst}
\usepackage[mathscr]{eucal}
\usepackage{amssymb,amsmath,amsfonts}
\usepackage{fancybox,fancyhdr}
\usepackage{graphicx}

\let\8=\infty \let\0=\emptyset  
\let\hat=\widehat

\let\tilde=\widetilde

\let\landa=\lambda
\let\alfa=\alpha

\let\parc=\partial

\def\landa{\lambda}

\def\flecha{\rightarrow}

\def\cte.{\mathop{\rm cte.}\nolimits}

\def\L{\mathbb{L}}
\def\E{\mathbb{E}}

\def\R{\mathbb{R}}

\def\C{\mathbb{C}}
\def\D{\mathbb{D}}

\def\H{\mathbb{H}}
\def\S{\mathbb{S}}

\def\Nil{{\rm{Nil}_3}}
\def\Ek{\mathbb{E}^3 (\kappa,\tau)}
\def\Mk{\mathcal{M}^2(\kappa)}

\newcommand{\nil}{\mathrm{Nil}_3}
\newcommand{\sol}{\mathrm{Sol}_3}

\def\Sol{\rm{Sol}_3}
\def\psl{\rm{PSL}(2,\R)}

\def\hr{\H^2\times\R}

\newcommand\esc[2]{\langle{#1},{#2}\rangle}

 \newtheorem{defi}{Definition}[section]
 \newtheorem{teo}[defi]{Theorem}
 
 \newtheorem{cor}[defi]{Corollary}

 \newtheorem{rem}[defi]{Remark}

\newcommand{\beq}{\begin{equation}}
\newcommand{\eeq}{\end{equation}}
\newcommand{\bp}{\begin{proof}\mbox{ }\newline\mbox{ }\vspace{-.3cm}\\ \mbox{ }\hspace{.5cm} }

\title[Constant mean curvature surfaces in $3$-dimensional Thurston geometries]{Constant mean curvature surfaces in \\ $3$-dimensional Thurston geometries}
\author{Isabel Fern\'andez and Pablo Mira}

\contact[isafer@us.es]{Isabel Fernández, Universidad de Sevilla
(Spain)}

\contact[pablo.mira@upct.es]{Pablo Mira, Universidad Politécnica de
Cartagena (Spain)}

\begin{document}

\maketitle

\begin{abstract}  This is a survey on the global theory of constant mean curvature surfaces in
Riemannian homogeneous $3$-manifolds. These ambient $3$-manifolds
include the eight canonical Thurston $3$-dimensional geometries,
i.e. $\R^3$, $\H^3$, $\S^3$, $\H^2\times \R$, $\S^2\times \R$, the
Heisenberg space $\nil$, the universal cover of ${\rm PSL}_2 (\R)$
and the Lie group $\sol$. We will focus on the problems of
classifying compact CMC surfaces and entire CMC graphs in these
spaces. A collection of important open problems of the theory is
also presented. \end{abstract}


\section{Introduction}

Constant mean curvature (CMC) surfaces appear as critical points of
a natural geometric variational problem: to minimize surface area
with or without a volume constraint (the unconstrained case
corresponds to zero mean curvature, i.e. to minimal surfaces). A
fundamental problem of this discipline is the geometric study and
classification of CMC surfaces under global hypotheses like
compactness, completeness, properness or embeddedness. The study of
this problem for CMC surfaces in the model spaces $\R^3$, $\S^3$ and
$\H^3$ has produced a very rich theory, in which geometric arguments
interact with complex analysis, harmonic maps, integrable systems,
maximum principles, elliptic PDEs, geometric measure theory and so
on.

One of the most remarkable achievements of this field in the last
decade has been the extension of this classical theory to the case
of CMC surfaces in simply connected homogeneous $3$-dimensional
ambient spaces. Apart from $\R^3$, $\S^3$ and $\H^3$, these spaces
are the remaining five Thurston $3$-dimensional geometries (i.e.
$\H^2\times \R$, $\S^2\times \R$, the Heisenberg group $\nil$, the
universal covering of ${\rm PSL}_2 (\R)$ and the Lie group $\sol$),
together with $3$-dimensional Berger spheres and some other Lie
groups with left-invariant metrics (see Section 2).

It must be said here that there is an important number of
contributions regarding CMC surfaces in general Riemannian
$3$-manifolds (not even homogeneous), many of which deal for
instance with isoperimetric questions or with geometric consequences
derived from the \emph{stability operator} associated to the second
variation of the surface. The achievement in the case of homogeneous
ambient 3-spaces has been the construction of a very rich global
theory of CMC surfaces, analogous to the case of $\R^3$, $\S^3$ and
$\H^3$, with an emphasis on the geometric classification (up to
ambient isometries) of properly immersed or properly embedded CMC
surfaces. The fact that the ambient space is homogeneous, i.e. it
has the same local geometry at all points, makes this problem
extremely natural.

Our aim here is to present a survey on some fundamental aspects of
the global theory of CMC surfaces in homogeneous $3$-manifolds. We
do not plan, however, to give a systematic account of all important
results of this already broad theory, but to discuss some specific
problems at the core of it. Hence, there will be many important
results omitted, and we apologize in advance for that.

In order to explain the problems we shall be dealing with, let us distinguish between compact and non-compact CMC surfaces in these spaces.

In the case of compact CMC surfaces, three fundamental problems are the Alexandrov problem (i.e. to classify compact embedded CMC surfaces), the Hopf problem (i.e. to classify CMC spheres), and the isoperimetric problem (recall that isoperimetric regions on a Riemannian $3$-manifold are bounded by compact embedded CMC surfaces, but the converse is not always true). By classical results, round spheres constitute the solution to each of these three problems in the case of CMC surfaces in $\R^3$. One of our main objectives will be to explain what is known (and what is not known) for these problems in the broader context of CMC surfaces in homogeneous $3$-manifolds.

In the case of non-compact CMC surfaces, one of the basic problems is to study the properly embedded CMC surfaces of finite topology. A classical result in that direction is given by \emph{Bernstein's theorem:} planes are the only entire minimal graphs in $\R^3$. As in all Thurston $3$-dimensional geometries there is a natural notion of entire graph, it is an important problem of the discipline to solve the \emph{Bernstein problem} for CMC graphs,     i.e. to classify all entire CMC graphs in these $3$-dimensional ambient spaces. This will be our other main objective.

The theory of CMC surfaces in Thurston $3$-dimensional geometries
started to develop as a consistent unified theory after some pioneer
works by Harold Rosenberg, jointly with William H. Meeks
\cite{MeRo1,MeRo2,Ros} for the case of minimal surfaces in product
spaces, and  jointly with Uwe Abresch \cite{AbRo1,AbRo2} for the
case of CMC surfaces in homogeneous spaces with a $4$-dimensional
isometry group.

On one hand, Meeks and Rosenberg established many results on
complete minimal surfaces in $M^2\times \R$, what has guided a large
number of subsequent works in the field. A recent major contribution
in this sense is the Collin-Rosenberg theorem \cite{CoRo} on the
existence of harmonic diffeomorphims from $\C$ onto the hyperbolic
plane $\H^2$, obtained by constructing an entire minimal graph of
parabolic conformal type in $\H^2\times \R$.

On the other hand, Abresch and Rosenberg discovered a holomorphic
quadratic differential for CMC surfaces in these homogeneous spaces
with $4$-dimensional isometry group (the $\Ek$ spaces), and solved
the Hopf problem for them. The general integrability theory of CMC
surfaces in the homogeneous $\Ek$ spaces was then established by B.
Daniel \cite{Dan1}. The discovery by the authors of a harmonic Gauss
map into $\H^2$ for $H=1/2$ surfaces in $\H^2\times \R$ turned into
a series of papers by Daniel, Fernández, Hauswirth, Mira, Rosenberg,
Spruck \cite{FeMi1,Dan2,FeMi2,HRS,DaHa} in which the Bernstein
problem for CMC graphs of \emph{critical mean curvature} (including
minimal graphs in Heisenberg space $\nil$, see Section 6) was
solved. Very recently, the Hopf and Alexandrov problems for CMC
surfaces have been solved by Daniel-Mira and Meeks \cite{DaMi,Mee}
in the remaining Thurston $3$-dimensional geometry: the Lie group
$\sol$, whose isometry group is only $3$-dimensional.

We have organized this exposition as follows. In Section 2 we will
introduce the $3$-dimensional homogeneous ambient spaces. In Section
3 we will present the basic integrability equations by Daniel for
CMC surfaces in the homogeneous spaces $\Ek$, together with the
holomorphic Abresch-Rosenberg differential, and with some basic
definitions on stability of CMC surfaces. In Section 4 we will
discuss the Hopf, Alexandrov and isoperimetric problems in the
homogeneous spaces $\Ek$. Section 5 will be devoted to solving the
Hopf and Alexandrov problems in the eighth Thurston geometry, i.e.
the Lie group $\sol$. In Section 6 we will present the solution to
the Bernstein problem for entire graphs of critical CMC in the
homogeneous $\Ek$ spaces. Finally, in Section 7 we shall expose the
Collin-Rosenberg theorem on parabolic entire minimal graphs in
$\H^2\times \R$, together with some developments on the theory of
complete minimal surfaces of finite total curvature in $\H^2\times
\R$. Most sections finish with a selection of important open
problems. See \cite{Mee,DHM} for more open problems in the theory.

A more detailed introduction to the global theory of CMC surfaces in
homogeneous $3$-spaces can be found in the Lecture Notes by Daniel,
Hauswirth and Mira \cite{DHM}.

The authors are grateful to H. Rosenberg, B. Daniel and J.A. Gálvez for useful observations about this manuscript.

\section{Homogeneous $3$-spaces and Thurston geometries}

Homogeneous spaces are the natural generalization of space forms. By
definition, a manifold is said to be homogeneous if the isometry
group acts transitively on the manifold. Roughly speaking, the
manifold looks the same at all the points, even though, standing at
one point, the manifold can look different in different directions.
In the simply connected case, the classification of the
$3$-dimensional homogeneous spaces is well-known. It turns out that
any simply connected homogeneous $3$-space must have isometry group
of dimension 6, 4 or 3. The complete list of these spaces is the
following (see subsections below for more details):
\begin{itemize}
\item The spaces with $6$-dimensional isometry group are  the space forms:
the Euclidean space $\R^3$,  the hyperbolic space $\H^3(\kappa)$,
and the standard sphere $\S^3(\kappa)$. For simplicity we will
assume that $\kappa=\pm 1$ and write $\H^3=\H^3(-1)$ and
$\S^3=\S^3(1)$.
\item  The spaces with $4$-dimensional isometry group are fibrations over the $2$-dimensional space
forms. They are the product spaces $\H^2\times\R$ and
$\S^2\times\R$, the Berger spheres, the Heisenberg space $\Nil$ and
the universal covering of the Lie group $\psl$.
\item  The spaces with $3$-dimensional isometry group are a certain class of Lie groups; among them we
specially quote the space $\Sol$.
\end{itemize}

These spaces are closely related with Thurston's Geometrization
Conjecture. This recently proved conjecture states that any compact
orientable $3$-manifold can be cut by disjoint embedded $2$-spheres
or tori into pieces, each one of them, after gluing $2$-balls or
solid tori along its boundary components, admits a geometric
structure. A $3$-manifold without boundary is said to admit a
geometric structure if it can be endowed with a complete locally
homogeneous metric. In this case, by considering its universal
covering we obtain a complete simply-connected locally homogeneous
space and hence, by a result of Singer, homogeneous. Thus, a
$3$-manifold admitting a geometric structure can be realized as the
quotient of a homogeneous simply connected $3$-space under the
action of a subgroup of a Lie group acting transitively by
isometries. The list of the maximal geometric structures that give
compact quotients consists of eight of the previously described
spaces: the three space forms, the two product spaces, $\Nil$, the
universal covering of ${\psl}$ and $\Sol$ (Berger spheres must be
excluded from this list because they are not maximal, their isometry
group are contained in the one of the standard sphere $\S^3$). We
refer to \cite{scott,bon} for more details.

\subsection{Homogeneous spaces with $4$-dimensional isometry group}
Denote by $\Mk$ the $2$-dimensional space form of constant curvature
$\kappa$ (for example, $\Mk=\R^2,\H^2,\S^2$ for $\kappa=0,-1,1$
respectively). As commented above, any simply connected homogeneous
$3$-space with $4$-dimensional isometry group admits a fibration
over $\Mk$, for some $\kappa\in\R$.  Moreover, these spaces can be
parameterized in terms of the base curvature $\kappa$ and the bundle
curvature $\tau$, that satisfy $\kappa - 4\tau^2\neq 0$. We will use
the notation $\Ek$ for these homogeneous spaces.

\begin{enumerate}
\item When $\tau=0$, we have the product spaces $\Mk\times\R$, i.e. up to scaling,
the spaces $\S^2\times\R$ when $\kappa>0$, and $\H^2\times\R$ when
$\kappa<0$.

\item When $\tau\neq 0$ and $\kappa>0$, the corresponding spaces are the Berger spheres, a family of $2$-parameter ($1$-parameter after a homothetical change of coordinates) metrics on the sphere, obtained by deforming the standard metric in such a way that the Hopf fibration is still a Riemannian fibration. They can also be seen as the Lie group $\rm{SU}(2)$ endowed with a $1$-parameter family of left-invariant metrics.

\item When $\tau \neq 0$ and $\kappa =0$, $\Ek$ is the Heisenberg group $\Nil$, the nilpotent Lie group
$$\left\{\begin{pmatrix}
1 & a & b\\
0 & 1 & c\\
0 & 0 & 1\end{pmatrix}\;;\; a,b,c\in\R\right\} ,$$ endowed with a
$1$-parameter family of left-invariant metrics, all of them
isometrically equivalent after a homothetical change of coordinates.

\item When $\tau\neq 0$ and $\kappa<0$, we obtain the universal covering of the Lie group $\psl$, endowed with a $2$-parameter (again $1$-parameter after homotheties) family of left-invariant metrics.

\end{enumerate}

There exists a common setting for all these spaces. Indeed, label
$\D(\rho)=\{(x_1,x_2)\in\R^2\,;\,x_1^2+x_2^2<\rho^2\}$. Then, if
$\kappa=0$ (resp. $\kappa<0$), the space $\Ek$ can be viewed as
$\R^3$ (resp. $\D\left(2/\sqrt{-\kappa}\right)\times\R$) endowed
with the metric
\begin{equation}\label{eq:metric}
ds^2 = \lambda^2 (dx_1^2+ dx_2^2)+\big(\tau\lambda (x_2dx_1 -
x_1dx_2) + dx_3\big)^2, \hspace{0.5cm}
\lambda=\frac{1}{1+\frac{\kappa}{4}(x_1^2+x_2^2)}.
\end{equation} Also, for $\kappa>0$, $(\R^3,ds^2)$ corresponds to the universal cover of $\Ek$ minus one
fiber. In all cases, up to a homothetical change of coordinates we
can suppose without loss of generality that $\kappa-4\tau^2=\pm 1$.

The corresponding Riemannian fibration $\pi:\Ek\to\Mk$ is given here
by the projection on the first two coordinates. The unitary vector
field
$$ \xi=\dfrac{\partial}{\partial x_3}$$
is a Killing field tangent to the fibers of $\pi$, and will be
referred to as the {\em vertical field} of the space $\Ek$. It
satisfies the equation $$\hat\nabla_X \xi= \tau X\times \xi$$ for
all vector fields $X$ in $\Ek$. Here $\hat\nabla$ is the Levi-Civita
connection, $\times$ the cross product and $\tau$ the bundle
curvature (this is basically the definition of $\tau$).

A remarkable difference between the spaces $\Ek$ is that their
isometry group has four connected components in the case $\tau=0$,
and only two when $\tau\neq 0$. This follows from the fact that any
isometry in the product spaces can either preserve or reverse the
orientation of the base and the fibers independently, while in the
case $\tau\neq 0$ it can only either preserve or reverse both
orientations. In particular, reflections only exist in product
spaces.

Also, when $\tau\neq 0$ the spaces $\Ek$ are Lie groups, and if we
set $\sigma:=\frac{\kappa}{2\tau}$, an orthonormal frame of
left-invariant vector fields (called the {\em canonical frame}) is
given by
$$\begin{array}{l}
E_1= \lambda^{-1} \left(  \cos(\sigma x_3)\dfrac{\partial}{\partial x_1} + \sin(\sigma x_3)\dfrac{\partial}{\partial x_2}  \right) + \tau(x_1 \sin(\sigma x_3) -x_2 \cos(\sigma x_3)   )\dfrac{\partial}{\partial x_3},\\
E_2= \lambda^{-1} \left( - \sin(\sigma x_3)\dfrac{\partial}{\partial x_1} + \cos(\sigma x_3)\dfrac{\partial}{\partial x_2}  \right) + \tau(x_1 \cos(\sigma x_3) + x_2 \sin(\sigma x_3)   )\dfrac{\partial}{\partial x_3},\\
E_3=\xi=\dfrac{\partial}{\partial x_3}.
\end{array}$$

\subsection{Homogeneous spaces with $3$-dimensional isometry group}

Of all homogeneous spaces with $3$-dimensional isometry group,
$\Sol$ is specially important, since it is the only Thurston
geometry among them. We will now describe some aspects of this
space.

A useful representation of $\Sol$ is the space $\R^3$ with the
metric
$$ds^2=e^{2x_3} dx_1^2 + e^{-2x_3} dx_2^2 + dx_3^2,$$
that is left-invariant for the structure of Lie group given by
$$ (x_1,x_2,x_3)\cdot(y_1,y_2,y_3) = ( x_1+e^{-x_3}y_1,x_2+e^{x_3}y_2,x_3+y_3  ).$$
The following vector fields form an orthonormal left-invariant frame
$$E_1=e^{-x_3} \dfrac{\partial}{\partial x_1},\quad E_2= e^{x_3} \dfrac{\partial}{\partial x_2} ,\quad  E_3=\dfrac{\partial}{\partial x_3}. $$
The isometries in $\Sol$ are generated by the three $1$-parameter
groups of translations
$$\begin{array}{c}
(x_1,x_2,x_3)\mapsto (x_1+c,x_2,x_3), \qquad
(x_1,x_2,x_3)\mapsto (x_1,x_2+c,x_3), \\[.2cm]
(x_1,x_2,x_3)\mapsto  (e^{-c} x_1, e^{c} x_2,x_3+c),
\end{array}$$
and by the orientation reversing isometries fixing the origin
$$(x_1,x_2,x_3)\mapsto  (-x_1,x_2,x_3), \qquad (x_1,x_2,x_3)\mapsto  (x_2,-x_1,-x_3).$$
A remarkable fact is the existence of two canonical foliations,
namely
$$\mathcal{F}_1=\{x_1=\mbox{constant}\}, \qquad \quad \mathcal{F}_2=\{x_2=\mbox{constant}\},$$
whose leaves are totally geodesic surfaces isometric to the
hyperbolic plane $\H^2$. Reflections across any of these leaves are
orientation reversing isometries of $\sol$.


\section{CMC surfaces: basic equations}

In this section we present three important tools for our study. One
is the set of integrability equations of CMC surfaces in $\Ek$ by
Daniel \cite{Dan1}. Another one the \emph{Abresch-Rosenberg
differential}, a holomorphic quadratic differential geometrically
defined on any CMC surface in $\Ek$. The third one is a local
isometric correspondence for CMC surfaces in $\Ek$ via which one can
pass from one homogeneous space into another when studying CMC
surfaces \cite{Dan1}. Some notions about the stability operator of
CMC surfaces are also given.


\subsection{Integrability equations in $\Ek$} It is well known that
the Gauss-Codazzi equations are the integrability conditions of
surface theory in $\R^3$, $\S^3$ and $\H^3$. In other homogeneous
spaces, the situation is more complicated.

Let $\psi:\Sigma\to\Ek$ be an isometric immersion with unit normal
map $\eta$, and consider on $\Sigma$ the conformal structure given
by its induced metric via $\psi$. Associated to a conformal
parameter $z=s+it$ on $\Sigma$, we will consider the usual operators
$\parc_z = (\parc_s -i\parc_t )/2$ and $\parc_{\bar{z}} =(\parc_s +i
\parc_t)/2$. Also denote by $\xi$ the vertical Killing field of
$\Ek$.

\begin{defi}\label{fundada} We call the \emph{fundamental
data} of $\psi$ the $5$-tuple $(\lambda |dz|^2,u,H,p \, dz^2 ,A
\,dz)$ where $H$ is the mean curvature and
$$\lambda=2\esc{\psi_z}{\psi_{\bar z}}, \hspace{1cm}
u=\esc{N}{\xi}, \hspace{1cm} p=-\esc{\psi_z}{N_z}, \hspace{1cm}
A=\esc{\xi}{\psi_z}.$$
\end{defi}
The function $u$ is commonly called the \emph{angle function} of the
surface.

Once here, a set of necessary and sufficient conditions for the
integrability of CMC surfaces in $\Ek$ can be written in terms of
these fundamental data. This is a result by B. Daniel \cite{Dan1},
although the formulation that we expose here (i.e. in terms of a
conformal parameter on the surface) comes from \cite{FeMi2}.

\begin{teo}[\cite{Dan1,FeMi2}]\label{th:formulas}
The fundamental data of an immersed surface $\psi:\Sigma\flecha \Ek$
satisfy the following integrability conditions:

\begin{equation}\label{lasces}
\left\{\def\arraystretch{1.3} \begin{array}{lccc} {\bf (C.1)} &
p_{\bar{z}} & = & \displaystyle \frac{\landa}{2} (H_z + u A (\kappa
- 4\tau^2)). \\ {\bf (C.2)} & A_{\bar{z}} & = & \displaystyle
\frac{u \landa}{2} (H+i\tau) .\\ {\bf (C.3)} & u_{z}
& = & - (H-i\tau) A -\displaystyle \frac{2 p}{\landa} \bar{A}.\\
{\bf (C.4)} & \displaystyle \frac{4 |A|^2}{\landa} & = & 1 - u^2 .
\end{array}\right.
\end{equation}

Conversely, if $\Sigma$ is simply connected, these equations are
also sufficient for the exis\-tence of a surface $\psi
:\Sigma\flecha \Ek$ with fundamental data $(\landa |dz|^2,u,H,p\,
dz^2,A\, dz)$. This surface is unique up to ambient isometries
preserving the orientations of base and fiber of $\Ek$.
\end{teo}

We see then that, in the spaces $\Ek$, more equations apart from the
Gauss-Codazzi ones are needed, due to the loss of symmetries. As a
matter of fact, {\bf (C.1)} is the Codazzi equation, while the Gauss
equation does not appear (it is deduced from the rest). These new
equations evidence the special character of the vertical direction
in the $\Ek$ spaces.

\begin{defi}
The \emph{Abresch-Rosenberg} differential of the immersion is
defined as the quadratic differential on $\Sigma$ given by
$$Q dz^2 = \left( 2 (H+i\tau) p - (\kappa -4 \tau^2) A^2 \right) dz^2.$$
\end{defi}
It is then easy to see by means of ${\bf (C.2)}$ that the Codazzi
equation {\bf (C.1)} can be rephrased in terms of $Q$ as
\begin{equation}\label{cod}
Q_{\bar{z}}= \landa H_z + (\kappa - 4\tau^2 ) \frac{H_{\bar{z}}
A^2}{(H+i\tau)^2}.
\end{equation}
Consequently, one has the following theorem, which generalized the
classical fact that the Hopf differential is holomorphic for CMC
surfaces in $\R^3$, $\S^3$ and $\H^3$.
 \begin{teo}[\cite{AbRo1,AbRo2}]\label{abro}
 $Q dz^2$ is a holomorphic quadratic differential on any CMC surface in $\Ek$.
  \end{teo}
This is a crucial result of the theory, since it allows the use of
holomorphic functions in the geometric classification of CMC
surfaces in $\Ek$ (see Section 4 and Section 6, for instance).


An important tool in the description of CMC surfaces in $\R^3$,
$\S^3$ and $\H^3$ is the classical Lawson correspondence. It
establishes an isometric one-to-one local correspondence between CMC
surfaces in different space forms that allows to pass, for instance,
from minimal surfaces in $\R^3$ to $H=1$ surfaces in $\H^3$.

The Lawson correspondence was generalized by B. Daniel to the
context of homogeneous spaces. Indeed, Daniel discovered in
\cite{Dan1} an isometric local correspondence for CMC surfaces in
all the homogeneous spaces $\Ek$, which can be described as follows
in terms of the \emph{fundamental data} defined above.

\begin{teo}[Sister correspondence, \cite{Dan1}]\label{th:sister}
Let $(\landa |dz|^2,u,H_1,p_1\, dz^2 ,A_1\, dz)$ be the fundamental
data of a simply connected $H_1$-CMC surface in
$\E(\kappa_1,\tau_1)$, and consider $\kappa_2,\tau_2,H_2 \in \R$ so
that
$$
\kappa_2 -4\tau_2^2 = \kappa_1 -4\tau_1^2, \qquad H_2^2 + \tau_2^2 =
H_1^2 +\tau_1^2.
$$
Then if we set $\theta \in \R$ given by $H_2 - i\tau_2 = e^{i\theta}
(H_1-i\tau_1)$, the fundamental data given by
\begin{equation}\label{fundadan}
(\landa |dz|^2, u, H_2, p_2 \, dz^2= e^{-i\theta} p_1 \, dz^2, A_2
\, dz =e^{-i\theta} A_1\, dz)
\end{equation}
give rise to a (simply connected) $H_2$-CMC surface in
$\E^3(\kappa_2,\tau_2)$, which is locally isometric to the original
one.
\end{teo}
Two surfaces related by the above correspondence are called {\em
sister surfaces} with phase $\theta$. In particular, the
corresponding Abresch-Rosenberg differentials of sister surfaces are
related by $Q_2=e^{-2i\theta}Q_1$. As special cases of this
correspondence we obtain the associate family of minimal surfaces in
$\Mk\times\R$, and a correspondence between minimal surfaces in
$\Nil$ and CMC $\frac{1}{2}$ surfaces in $\H^2\times\R$.
Generically, and up to ambient isometries and dilations, the family
of sister surfaces for a given choice of $(H,\kappa,\tau)$ is a
continuous $1$-parameter family.

There is a natural notion of {\em graph} in these spaces. Since
$\Ek$ has a canonical fibration over $\Mk$ (see Section 2), we will
say that an immersed surface $\Sigma$ in $\Ek$ is a (local) graph if
the projection to the base is a (local) diffeomorphism. The
CMC-equation for the graph of a function $u=u(x,y)$ is the PDE (see
\cite{lee})
\begin{equation}\label{eq:graphs}
 \dfrac{2H}{\delta^2}=
\dfrac{\partial}{\partial x} \left( \dfrac{\alpha}{\omega} \right) +
\dfrac{\partial}{\partial y} \left(  \dfrac{\beta}{\omega} \right),
\end{equation}
where
$$\begin{array}{cc}
\delta= 1+\dfrac{\kappa}{4}(x^2+y^2), &
\omega=\sqrt{ 1+ \delta^2(x^2+y^2)},\\[.2cm]
\alpha = u_x + \tau \dfrac{y}{\delta}, & \beta = u_y -
\tau\dfrac{x}{\delta}.
\end{array}$$
For instance, a graph $u=u(x,y)$ in $\Nil \equiv \E^3(0,\tau)$ is
minimal if and only if it satisfies the elliptic PDE

 \begin{equation}\label{min}
(1+ \beta^2) u_{xx}  - 2 \alfa \beta \, u_{xy} + (1+\alfa^2)u_{yy}
=0,
 \end{equation}
where $\alfa:= u_x + y/2$ and $\beta := u_y - x/2$.

\subsection{Stability and index of CMC surfaces}

As it is well known, CMC surfaces in Riemannian $3$-manifolds appear
as the critical points for the area functional associated to
variations of the surface with compact support and constant enclosed
volume. Equivalently, an immersed surface $S$ has constant mean
curvature $H$ if and only if it is a critical point for the
functional ${\rm Area} - 2H\,{\rm Vol}$. The second variation
formula for this functional is given by
$$Q(f,f)=- \int_S f{\mathcal{L}}(f),$$
where ${\mathcal{L}}$ is the \emph{Jacobi operator} (or
\emph{stability operator}) of the surface:
$${\mathcal{L}}= \Delta + ||{\mathcal{B}}||^2 + \mbox{Ric}(\eta).$$
Here $\Delta$ is the Laplacian for the induced metric on the
surface, ${\mathcal{B}}$ is the second fundamental form, $\eta$ is
the unit normal vector field, and $\mbox{Ric}$ is the Ricci
curvature in the ambient manifold. As a particular case, the Jacobi
operator for CMC surfaces in the spaces $\Ek$ can be rewritten (see
\cite{Dan1}) as
$${\mathcal{L}}= \Delta - 2K + 4H^2 + 4\tau^2 + (\kappa-4\tau^2)(1+u^2),$$
being $K$ the Gaussian curvature of the surface and $u$ the angle
function (see Definition \ref{fundada}). A {\em Jacobi function} is
a function $f$ for which $\mathcal{L}(f)=0$.

A CMC surface $S$ is said to be {\em stable}  (resp. {\em weakly
stable}) if
$$Q(f,f)= - \int_S f{\mathcal{L}}(f) \geq 0$$
holds for any smooth function $f$ on $S$ with compact support (resp.
with compact support and $\int_S f=0$).  For instance, CMC graphs in
$\Ek$ are stable, and compact CMC surfaces bounding isoperimetric
regions are weakly stable (but not necessarily stable, as round
spheres in $\R^3$ show).

An important concept related to stability is the {\em index} of a
CMC surface. The index of a compact CMC surface is defined as the
number of negative eigenvalues of its Jacobi operator. Thus, stable
CMC surfaces (in particular CMC graphs) have index zero. Round
spheres in $\R^3$ have index one.

We refer to \cite{stable} for more details about stability of CMC
surfaces.


\section{Compact CMC surfaces in $\Ek$}

In this section we explain the most important results that are known
regarding the existence and uniqueness of compact CMC surfaces in
the homogeneous $3$-spaces $\Ek$. The fundamental examples are the
rotational CMC spheres, and we shall be interested in their
uniqueness among compact embedded CMC surfaces, and among immersed
CMC surfaces. These problems are called, respectively, the
Alexandrov and Hopf problems.

\subsection{Rotational compact CMC surfaces}

Although round spheres in the model spaces $\R^3,\S^3,\H^3$ are CMC
spheres, this does not hold for the rest of homogeneous spaces.
However, in all the spaces $\Ek$ there exist rotations with respect
to the vertical axis, and so there is a natural notion of
\emph{rotational surface}. It is hence natural to seek CMC spheres
(and CMC tori) in $\Ek$ among the class of rotational surfaces. This
can be done by ODE analysis, and the result of this can be
summarized as follows:

\begin{teo} (Structure of rotational CMC spheres in $\Ek$).
\begin{enumerate}
\item If $\kappa-4\tau^2>0$, then for every $H\in\R$ there exists a unique rotational CMC $H$ sphere (up to isometries) in
$\E^3(\kappa,\tau)$. These spheres are embedded if $\tau=0$, i.e. in
$\S^2\times \R$, and also for most Berger spheres. However, for some
Berger spheres with small bundle curvature $\tau$ (with respect to a
fixed $\kappa$) there is a certain region of variation of the
parameters $(H,\tau)$ where the spheres are non-embedded. This
region can be explicitly described, see \cite{Tor}.
\item If $\kappa-4\tau^2<0$, then
\begin{itemize}
\item if $H^2\leqslant-\frac{\kappa}4$, then there exists no rotational CMC $H$ sphere in $\E^3(\kappa,\tau)$,
\item if $H^2>-\frac{\kappa}4$, then there exists a unique rotational CMC $H$ sphere (up to isometries) in
$\E^3(\kappa,\tau)$. All these spheres are embedded.
\end{itemize}
\end{enumerate}
\end{teo}

Let us remark that all these CMC spheres can be constructed
explicitly. We shall call them \emph{canonical rotational CMC
spheres}. For example, the rotational CMC $H$ spheres in $\S^2\times
\R\subset \R^4$ are given by the formula

$$\psi (u,v)= (-\cos k(u), \sin k(u) \cos v, \sin
k(u) \sin v, h(u)),$$ where $-1\leq u\leq 1$, $H\in \R$ and $$k(u):=
2 \, {\rm arc tan} \left(\frac{2H}{\sqrt{1-u^2}}\right),
\hspace{1cm} h(u):= \frac{4H}{\sqrt{4H²+1}} \, {\rm arc sinh }
\left(\frac{u}{\sqrt{1-u^2+4H^2}}\right).$$ Besides these rotational
CMC spheres, there also exist rotational CMC tori in $\Ek$ when (and
only when) $\kappa -4\tau^2 >0$ (excluding minimal surfaces in
$\S^2\times \R$). For $\S^2\times \R$, they are all embedded (see
Pedrosa \cite{Ped}). For Berger spheres the situation is explained
by Torralbo and Urbano in \cite{Tor,ToUr}; one has for every $H$
rotational embedded CMC tori given by the Hopf lift of a circle in
$\S^2$, but there also exist some other non-flat rotational CMC
tori. The embeddedness problem for such tori is open in general, but
for the minimal case there are embedded rotational tori other than
Clifford tori. This contrasts with the case of embedded minimal tori
in $\S^3$.

A general study of CMC surfaces in $\H^2\times \R$ and $\S^2\times
\R$ invariant by a continuous $1$-parameter subgroup of ambient
isometries can be found in \cite{SaE,SaTo}.

\subsection{The Alexandrov problem in $\Ek$}

One of the fundamental theorems of CMC surface theory is the
so-called \emph{Alexandrov theorem}.

\begin{teo}[Alexandrov]
Any compact embedded CMC surface in $\R^3$, $\H^3$ or a hemisphere
of $\S^3$ is a round sphere.
\end{teo}
\begin{proof}
The proof relies on the so-called Alexandrov reflection principle,
which we sketch for $\R^3$ although it works with great generality.
Consider a plane $P$ disjoint from the compact embedded CMC surface
$\Sigma$, and start translating it in a parallel way towards
$\Sigma$. After it first touches $\Sigma$, we start reflecting the
piece of $\Sigma$ that has been left behind across this new
translated plane. In this way we will eventually reach a first
contact point with the unreflected part of $\Sigma$. By the maximum
principle for elliptic PDEs, this means that $\Sigma$ is symmetric
with respect to such a plane. As the starting plane was arbitrary,
the compact surface must be a round sphere.
\end{proof}

It must be emphasized that there exist embedded CMC tori in $\S^3$,
such as the product tori $\S^1 (r)\times \S^1 (\sqrt{1-r^2})\subset
\S^3$. Thus, the hemisphere hypothesis is necessary in the case of
$\S^3$.

Motivated by this result, the problem of
classifying all compact embedded CMC surfaces in a Riemannian $3$-manifold $\bar{M}^3$
will be called the \emph{Alexandrov
problem} in $\bar{M}^3$.

In the case of CMC surfaces in the product spaces $\H^2\times \R$
and $\S^2\times \R$, the Alexandrov technique can be applied for
\emph{horizontal} directions, and so the following result holds.

\begin{teo}[Hsiang-Hsiang]
Any compact embedded CMC surface in $\H^2\times \R$ or $\S_+^2\times
\R$ is a standard rotational CMC sphere.
\end{teo}
Again, the hemisphere hypothesis is necessary, since we know that
there are embedded CMC tori in $\S^2\times \R$.

As regards the homogeneous spaces $\Ek$ with $\tau\neq 0$, i.e.
Heisenberg space, Berger spheres and the universal covering of ${\rm
SL}_2 (\R)$, the Alexandrov problem is open. The main difficulty
there is that these spaces do not admit reflections, and hence the
reflection principle does not hold.

\subsection{The Hopf problem in $\Ek$}

Another fundamental result of CMC surface theory is the \emph{Hopf
theorem}:

\begin{teo}[Hopf]
Any immersed CMC sphere in $\R^3$, $\S^3$ or $\H^3$ is a round
sphere.
\end{teo}
\begin{proof}
The Hopf differential (see Section 3) of any CMC surface in $\R^3$,
$\S^3$ or $\H^3$ is holomorphic, and vanishes at the umbilical
points of the surface. As any holomorphic quadratic differential
must vanish on the Riemann sphere, we conclude that immersed CMC
spheres are totally umbilical, and hence round spheres.
\end{proof}

The \emph{Hopf problem} in a Riemannian $3$-manifold
$\bar{M}^3$ will refer to the problem of classifying all immersed CMC spheres
in $\bar{M}^3$.

As was proved in Theorem \ref{abro}, CMC surfaces in the homogeneous
spaces $\Ek$ have an associated holomorphic quadratic differential:
the Abresch-Rosenberg differential $Q_{AR}$. This allows to solve
the Hopf problem in $\Ek$, along the lines suggested by Hopf's
classical theorem. We present here an alternative proof to the
original one by Abresch and Rosenberg \cite{AbRo1,AbRo2}, based on
Daniel's integrability equations, and on some ideas in
\cite{FeMi2,GMM} (see \cite{dCF,EsRo2,DHM}).

\begin{teo}[Abresch-Rosenberg]
Any immersed CMC sphere in $\Ek$ is a standard rotational sphere.
\end{teo}
\begin{proof}
As the Abresch-Rosenberg $Q_{AR}$ is holomorphic, it must vanish on
any immersed CMC sphere. So, we need to prove that spheres with
$Q_{AR}=0$ are rotational.

First, one can observe that on any CMC surface in $\Ek$, the
equation $Q_{AR}=0$ together with the integrability conditions in
Theorem \ref{th:formulas} imply that the function $w:= {\rm arctanh}
(u)$ is a harmonic function on the surface (here $u$ is the angle
function of the surface). So, once we rule out the case $u = {\rm
const.}$ which does not produce CMC spheres (except for slices in
$\S^2\times \R$), we can define $\zeta$ to be a local conformal
parameter on the surface with ${\rm Re} \, \zeta =w$. Again from the
integrability equations ${\bf (C.1)}$ to ${\bf (C.4)}$ we see that
all fundamental data of the surface depend only on $w$ (and not on
${\rm Im} \, \zeta$). This implies that the surface is a local piece
of some CMC surface invariant by a continuous $1$-parameter subgroup
of ambient isometries of $\Ek$.

If the surface is compact, this isometry subgroup must be the group
of rotations around the vertical axis, with the possible exception
of Berger $3$-spheres (the only space in which there are
non-rotational compact continuous isometry subgroups). However, it
is clear that any element of such a non-rotational isometry subgroup
has no fixed points. Hence, by the invariance property, there is a
globally defined non-zero vector field on the surface (that is
tangent to the orbits). But this is impossible on a sphere. Hence,
the isometry subgroup is always the group of rotations around the
vertical axis, and thus the CMC sphere is rotational, as wished.
\end{proof}

\subsection{The isoperimetric problem in $\Ek$}

The Alexandrov problem is very relevant to the isoperimetric problem
in a Riemannian $3$-manifold $\bar{M}^3$; indeed, any solution to
the isoperimetric problem in $\bar{M}^3$ is a region bounded by a
compact embedded CMC surface. So, for instance, the only candidates
to solve the isoperimetric problem for a given volume in $\H^2\times
\R$ are rotational CMC spheres. Another geometric property satisfied
by isoperimetric solutions is that they are weakly stable, see
Section 2.

The class of isoperimetric solutions in $\R^3$, $\S^3$ and $\H^3$ is
the class of round spheres. The isoperimetric problem in $\S^2\times
\R$ and $\H^2\times \R$ has been also explicitly solved, as follows:

\begin{enumerate}
  \item
The isoperimetric regions in $\H^2\times \R$ are exactly the regions
bounded by the canonical rotational CMC spheres. (Hsiang-Hsiang).
 \item
There is a value $H_1\approx 0.33$ such that the isoperimetric
regions in $\S^2\times \R$ are exactly the regions bounded by the
canonical rotational CMC $H$ spheres with $H\geq H_1$. (Pedrosa).
\end{enumerate}
So, regading complete simply connected Riemannian $3$-manifolds, the
isoperimetric problem is fully solved in $\R^3$, $\S^3$, $\H^3$,
$\S^2\times \R$ and $\H^2\times \R$. A remarkable advance in this
direction has been obtained very recently by F. Torralbo and F.
Urbano \cite{ToUr}, who have added to this list a certain subfamily
of Berger spheres:

\begin{teo}[Torralbo-Urbano]\label{torur}
The solutions to the isoperimetric problem in the Berger spheres
$\Ek$ with $\frac{1}{3}\leq \frac{4\tau^2}{\kappa} <1$ are the
canonical rotational CMC spheres.
\end{teo}
The proof of this result relies on embedding the Berger spheres
$\Ek$ into the $4$-dimensional complex space $\C P^2$, and using a
Willmore inequality in this space due to Montiel and Urbano
\cite{MoUr}.

For the rest of the spaces, the isoperimetric problem is open. In
any case, the general theory of the isoperimetric problem together
with the Abresch-Rosenberg uniqueness theorem imply that, for small
volumes, the isoperimetric solutions are canonical rotational CMC
spheres with large $H$.

\subsection{Open problems}
One of the major unsolved problems in the theory is the Alexandrov
problem when $\tau \neq 0$, i.e. in $\Nil$, the universal cover of
$\psl$ and Berger hemispheres. It is conjectured that canonical
rotational spheres are the only compact embedded CMC surfaces in
these spaces. A related open problem is the isoperimetric problem in
$\Nil$, the universal cover of $\psl$ and the Berger spheres not
covered by Theorem \ref{torur}. In the first two cases, it is
conjectured that the isoperimetric solutions are exactly the
canonical rotational spheres.

Besides, it is conjectured by Nelli and Rosenberg \cite{NeRo2} that
compact weakly stable CMC surfaces in $\H^2\times \R$ are rotational
CMC spheres.

Another important problem of the theory is the construction of
higher genus compact (immersed) CMC surfaces, (e.g. CMC tori) in
$\Ek$ with $\kappa\leq 0$.

\section{CMC spheres in $\sol$}

In this section we will expose the recent solution to the Alexandrov
problem (i.e. the classification of compact embedded CMC surfaces)
and the Hopf problem (i.e. the classification of immersed CMC
spheres) in the remaining Thurston $3$-geometry: the homogeneous
space $\sol$.

The first step in this direction is that we can solve the Alexandrov
problem \emph{from a topological point of view}.

\begin{teo}[Rosenberg]\label{rosenberg}
Any compact embedded CMC surface in $\sol$ is, topologically, a
sphere.
\end{teo}
\begin{proof}
By Alexandrov reflection principle using the two canonical
foliations of $\sol$ (recall that reflections across their leafs are
orientation-reversing isometries of $\sol$), it turns out that any
compact embedded CMC surface in $\sol$ is a bi-graph with respect to
two linearly independent directions in $\R^3$. Thus, the surface is,
topologically, a sphere.
\end{proof}

This result leaves us with the problem of classifying (embedded) CMC
spheres. A substantial difficulty for this task is that $\sol$ has
no rotations. Hence, there are no rotational CMC spheres to use in
order to gain insight of the theory, and even the existence of CMC
spheres for a given value of $H$ needs to be settled.

The next theorem is the main result of the section, and solves the
Hopf and Alexandrov problems in $\sol$.

\begin{teo}[Daniel-Mira, Meeks]\label{mainsol}
For every $H>0$ there exists an embedded CMC $H$ sphere $S_H$ in
$\sol$. This sphere is unique in the following sense:
 \begin{enumerate}
   \item
\emph{Hopf uniqueness:} every immersed CMC $H$ sphere in $\sol$ is a
left-translation of $S_H$.
 \item
\emph{Alexandrov uniqueness:} every compact embedded CMC $H$ surface
in $\sol$ is a left-translation of $S_H$.
 \end{enumerate}
Moreover, each sphere $S_H$ has index one, it inherits all possible
symmetries of the ambient space (its group of ambient isometries is
the diedral group $D_4$), its Lie group Gauss map is a
diffeomorphism, and the family $\{S_H: H>0\}$ is real analytic (up
to left translations).
\end{teo}
\begin{rem}
Theorem \ref{mainsol} was obtained by Daniel and Mira \cite{DaMi}
for $H>1/\sqrt{3}$. For the remaining values $H\in (0,1/\sqrt{3}]$,
Daniel and Mira also proved the uniqueness in the Hopf and
Alexandrov sense for all values of $H$ for which there exists an
index one CMC $H$ sphere. Finally, Meeks \cite{Mee} obtained the
existence of index one CMC $H$ spheres for every $H>0$ (and not just
for $H>1/\sqrt{3}$). This concluded the proof of Theorem
\ref{mainsol}.
\end{rem}

We shall split the sketch of the proof of Theorem \ref{mainsol} into
two parts.

\subsection{Proof of Theorem \ref{mainsol}: uniqueness}

The results of this part are contained in \cite{DaMi}. The Lie group
Gauss map $g:\Sigma\flecha \overline{\C}$ of a CMC surface
$X:\Sigma\flecha \sol$ satisfies the following elliptic PDE (here
$z$ is a conformal parameter on the surface):

\begin{equation} \label{eqg2}
g_{z\bar z}=A(g)g_zg_{\bar z}+B(g)g_z\bar g_{\bar z},
\end{equation}
where, by definition,
\begin{equation} \label{defAB2}
A(q)=\frac{R_q}R=\frac{2H(1+|q|^2)\bar q+2q}{R(q)},\quad
B(q)=\frac{R_{\bar q}}R-\frac{\bar R_{\bar q}}{\bar R}=-\frac{4H(1+|q|^2)(\bar q+q^3)}{|R(q)|^2},
\end{equation}
$$R(q)=H(1+|q|^2)^2+q^2-\bar q^2.$$
Moreover, the surface $X$ is uniquely determined by the Gauss map $g$, and it can actually be recovered from $g$ by means of an integral representation formula.

Once here, the first idea in order to prove a Hopf-type theorem is
to look for a holomorphic quadratic differential for CMC surfaces in
$\sol$. However, it seems that such a holomorphic object is not
available in the theory; this constitutes another key difference
from the theory of CMC surfaces in the other Thurston $3$-geometries
exposed in the previous section, where the Abresch-Rosenberg (or the
Hopf differential) is holomorphic.

Still, it is not strictly necessary to obtain a holomorphic differential in order to prove a Hopf-uniqueness theorem: it suffices to find a geometrically defined quadratic differential with isolated zeros of negative index, so that it vanishes identically on spheres. This is done as follows.

\begin{teo}[Daniel-Mira]
Let $H>0$, and assume that there exists an index one CMC $H$ sphere $S_H$ in $\sol$. Then there exists a quadratic differential $Q_H$, geometrically defined on any CMC $H$ surface in $\sol$, with the following properties:
 \begin{enumerate}
 \item
It only has isolated zeros of negative index (thus, it vanishes on spheres).
 \item
$Q_H=0$ holds for a surface $X:\Sigma\flecha \sol$ if and only if $X$ is a left-translation of some piece of the sphere $S_H$.
 \end{enumerate}
Moreover, the sphere $S_H$ is embedded, and it is therefore unique
in $\sol$ (up to left-translations) in the Hopf sense and in the
Alexandrov sense.
\end{teo}
The quadratic differential $Q_H$ is constructed as follows. Let
$G:S_H\equiv \overline{\C}\flecha \overline{\C}$ denote the Gauss
map of $S_H$. Then $G$ is a diffeomorphism (otherwise one can
construct a Jacobi function $u$ on $S_H$ with $u(p)=\nabla u(p)=0$
at some $p\in S_H$, which contradicts the index one condition by
Courant's nodal domain theorem).

Once here, the differential $Q_H$ is defined for any CMC $H$ surface
$X:\Sigma\flecha \sol$ with Gauss map $g:\Sigma\flecha
\overline{\C}$ by
\begin{equation}\label{qfor}
Q_H = ( L(g) g_z^2 + M(g) g_z \bar{g}_z )\, dz^2,
\end{equation}
where by definition
\begin{equation} \label{defM}
M(q)=\frac1{R(q)}=\frac1{H(1+|q|^2)^2+q^2-\bar q^2}
\end{equation}
and $L:\overline{\C}\flecha \C$ is implicitly given in terms of the
Gauss map $G$ of $S_H$ by
\begin{equation}\label{elege}
L(G(z))=-\frac{M(G(z))\bar G_z(z)}{G_z(z)}.
\end{equation}

It must also be emphasized that, by this uniqueness theorem, any index one CMC sphere $S_H$ in $\sol$ is as symmetric as the ambient space allows: there is a point $p\in \sol$ such that $S_H$ is invariant with respect to all the isometries of $\sol$ that leave $p$ fixed.

\subsection{Proof of Theorem \ref{mainsol}: existence}

Let us define $$\mathcal{I}:=\{H>0 : \text{exists an index one CMC
$H$ sphere $S_H$ in $\sol$}\}.$$ We prove next the theorem by Meeks \cite{Mee} that $\mathcal{I} =(0,\8)$. (The fact that $(1/\sqrt{3},\8)\subset \mathcal{I}$ had been previously obtained in \cite{DaMi}).

That $\mathcal{I}\neq \emptyset$ follows from the existence of
isoperimetric spheres, which in $\sol$ must have index one. That
$\mathcal{I}$ is open was proved in \cite{DaMi}, and follows from
the implicit function theorem and from the continuity of the
eigenvalues and eigenspaces in the deformation.

The proof that $\mathcal{I}$ is closed is the critical step. The key
point is to prevent that the diameters of a sequence of CMC $H_n$
spheres $(S_{H_n})$ with $H_n\to H_0 >0$ tend to $\8$. This was
proved first by Daniel-Mira, but only for $H_0>1/\sqrt{3}$. The
final proof for every $H_0>0$ was recently given by Meeks
\cite{Mee}, using the following height estimate: \emph{there exists
a constant $K(H_0)$ such that for any CMC $H_0$ graph (possibly
non-compact) with respect to one of the two canonical foliations of
$\sol$, and with boundary on a leaf, the maximum height attained by
the graph with respect to this leaf is $\leq K(H_0)$}.

Once this height estimate is ensured, Meeks concludes the proof by
some elliptic theory and stability arguments.

\subsection{Open problems}

Are CMC spheres in $\sol$ weakly stable? Do they all bound
isoperimetric regions in $\sol$? A positive answer is conjectured in
\cite{DaMi}. What happens in other homogeneous $3$-spaces with
$3$-dimensional isometry group?

It seems very interesting to develop a global theory of minimal
surfaces in $\sol$. Some natural problems would be proving
half-space theorems, classifying entire minimal graphs, or finding
properly embedded minimal surfaces of non-trivial topology.

\section{Surfaces of critical CMC}
As we saw in Section 3, CMC $H$ spheres in the homogeneous space
$\Ek$ exist exactly for the values $H^2 >-\kappa/4$. Besides, one
can easily see that there exist entire rotational CMC $H$ graphs in
$\Ek$, $\kappa\leq0$, whenever $H^2\leq -\kappa /4$. From these
results and the maximum principle, we obtain

\begin{teo}
Any compact CMC $H$ surface in $\Ek$ satisfies $H^2 >-\kappa/4$.
Also, any entire CMC graph in $\Ek$, $\kappa\leq 0$, satisfies $H^2
\leq -\kappa/4$.
\end{teo}

There are several other properties that make the theory of CMC
surfaces with $H^2>-\kappa /4$ quite different from the theory of
CMC surfaces with $H^2\leq -\kappa/4$. For instance:

\begin{enumerate}
  \item
A properly embedded CMC surface in $\H^2\times \R$ with $H>1/2$ and
finite topology cannot have exactly one end (Espinar, Gálvez,
Rosenberg, \cite{EGR}).
 \item
There exist horizontal and vertical height estimates for CMC
surfaces with $H>1/2$ in $\H^2\times \R$ \cite{NeRo2,AEG1,EGR}.
 \item
There are no complete stable CMC surfaces in $\H^2\times
\R$ with $H>1/\sqrt{3}$, and the result is expected for $H>1/2$.
(Nelli-Rosenberg, \cite{NeRo2}). Besides, there are no complete
stable CMC surfaces with $H>1/2$ in $\H^2\times \R$ of
parabolic conformal structure (Manzano-Pérez-Rodríguez,
\cite{MaPR}).
\end{enumerate}

It is hence natural to introduce the following definition.

\begin{defi}
We say that a CMC surface in $\Ek$ with $\kappa\leq 0$ has
\emph{critical} CMC if its mean curvature $H$ satisfies $H^2 =
-\kappa /4$.
\end{defi}

The critical mean curvature is the largest value of $|H|$ for which
compact CMC surfaces do not exist. Therefore we have $H=1/2$
surfaces in $\H^2\times \R$, minimal surfaces in $\nil$, and
$H=\sqrt{-\kappa}/2$ surfaces in the universal covering of $\psl$. A
remarkable property is that the sister correspondence preserves the
property of having critical CMC, and that every simply connected
surface of critical CMC is the sister surface of some minimal
surface in $\nil$.

In this section we will study the global geometry of surfaces with
critical CMC, focusing on the existence of harmonic Gauss maps and
the classification of entire graphs.

\subsection{Harmonic Gauss maps}

A smooth map $G:M\to N$ between Riemannian manifolds is harmonic if
it is a critical point for the total energy functional. When $M$ is
a surface, harmonicity is a conformal invariant, and it implies that
the quadratic differential
$$ Q_0 dz^2 = \esc{G_z}{G_z} dz^2, $$
is holomorphic, where $z$ is a conformal parameter on $\Sigma,$ and
$\esc{}{}$ denotes the metric in $N$ (see \cite{wood}). We call
$Q_0dz^2$ the {\em Hopf differential} associated to $G$.

The Gauss map of CMC surfaces in $\R^3$ is harmonic into $\S^2$, and
its Hopf differential agrees (up to a constant) with the Hopf
differential of the surface. Moreover, the CMC surface can be
recovered from the Gauss map by a representation formula. This Gauss
map opens the door to the use of strong techniques from harmonic
maps in the description of CMC surfaces.

The same holds for spacelike CMC surfaces in Minkowski $3$-space
$\L^3$, but this time the harmonic Gauss map takes values into
$\H^2$. Let us briefly comment this case, since it will play an
important role in the development of the section. Let
$f:\Sigma\to\L^3$ be a connected spacelike CMC surface, oriented so
that its Gauss map $G$ takes values in $\H^2$. Here $\L^3$ is $\R^3$
with the metric $dx^2 +dy^2 -dz^2$ and $\H^2$ is realized in $\L^3$
in the usual way. It turns out that $G$ is harmonic into $\H^2$ and
its associated Hopf differential agrees (up to a multiplicative
constant) with the Hopf differential of the immersion $f$. Moreover,
the metric of the CMC surface, $\esc{df}{df}=\tau_0|dz|^2$, is
related with $G$ by
$$2\esc{G_z}{G_{\bar{z}}} =\dfrac{\tau_0}{4} + \dfrac{4|Q_0|^2}{\tau_0}.$$

\begin{defi}\label{def:wei}
We will say that a harmonic map $G$ into $\H^2$ admits {\em
Weierstrass data} $\{Q_0,\tau_0\}$ if the pullback metric induced by
$G$ can be written as
$$\esc{dG}{dG} = Q_0dz^2 + \mu|dz|^2 + \bar{Q_0}d\bar{z}^2, \qquad
\mu=\dfrac{\tau_0}{4} + \dfrac{4|Q_0|^2}{\tau_0},$$ $\tau_0$ being a
positive smooth function.
\end{defi}


\subsubsection{$H=1/2$ surfaces in $\H^2\times\R$}

We will regard $\hr=\E^3(-1,0)$ in its Minkowski model, i.e.
$$ \H^2\times\R=\{(x_0,x_1,x_2,x_3): \, x_0>0, -x_0^2+x_1^2+x_2^2=-1\}\subset \L^3\times \R=\L^4.$$
Using this model, the unit normal vector $\eta$ of an immersed
surface $\psi=(N,h):\Sigma\to\H^2\times\R$ takes values in the de
Sitter $3$-space, and $\{\eta,N\}$ is an orthonormal frame for the
Lorentzian normal bundle of $\psi$ in $\L^4$. Moreover, if $u$ is
the angle function of the surface (that is, the last coordinate of
$\eta$) and we assume that $u\neq 0$ (that is, that $\psi$ is
nowhere vertical, or equivalently, that it is a multigraph), then we
can write
\begin{equation}\label{eq:gauss12}
\dfrac{1}{u} (\eta + N)=(G,1),
\end{equation}
for a certain map $G:\Sigma\to\H^2.$

\begin{defi}[\cite{FeMi1}]
The map $G$ given by \eqref{eq:gauss12} will be called the {\em
hyperbolic Gauss map} of an immersed (nowhere vertical) surface in
$\H^2\times\R$.
\end{defi}

The main property of the hyperbolic Gauss map is the following
\cite{FeMi1}:

\begin{teo}[Fern\'andez-Mira]\label{th:gauss12}
The hyperbolic Gauss map of a CMC surface with $H=1/2$ in
$\H^2\times\R$ is a harmonic map into $\H^2$, and admits Weierstrass
data $\{-Q,\lambda u^2\},$ where $Qdz^2$, $\lambda|dz|^2$ and $u$
are, respectively, the Abresch-Rosenberg differential, the metric,
and the angle function of the surface.

Conversely, if $\Sigma$ is simply connected, any harmonic map
$G:\Sigma\flecha \H^2$ admitting Weierstrass data is the hyperbolic
Gauss map of some $H=1/2$ surface in $\H^2\times \R$.

Moreover, the space of $H=1/2$ surfaces in $\H^2\times \R$ with the
same hyperbolic Gauss map $G$ is generically two-dimensional, and it
can be recovered from $G$ by a representation formula.
\end{teo}

The proof of the direct part of the above result follows from
equations \eqref{lasces} and the very definition of $G$. The
converse part is an integrability argument. This result is of great
importance for the rest of this section, since it allows the use of
harmonic maps in the description of surfaces of critical CMC.

\subsubsection{Minimal surfaces in $\Nil$}
The existence of this harmonic Gauss map for $H=1/2$ surfaces in
$\H^2\times \R$ was extended by B. Daniel \cite{Dan2} to the case of
minimal surfaces in $\Nil=\E^3(0,\frac{1}{2})$.

This time, the harmonic Gauss map is given by the Lie group Gauss
map of the surface. Indeed, if we identify the Lie algebra of $\Nil$
with the tangent space at a point by left multiplication, we can
stereographically project the unit normal vector field to obtain a
map taking values in the extended complex plane. More specifically,
we will consider the model of $\Nil$ given in Section 2 and its
canonical frame of left-invariant fields $\{E_1,E_2,E_3\}$. If
$N=\sum N_i \, E_i$ is the unit normal of $X:\Sigma\flecha \Nil$,
then the Gauss map of $X$ is given by
$$g= \frac{N_1+iN_2}{1+N_3}:\Sigma\flecha \overline{\C}.$$

Now, if the surface is nowhere vertical we can orient it so that
$u=\esc{N}{E_3}$ is positive, and so $g$ takes values in the unit
disc $\D$. By identifying $\H^2$ with $(\D,ds_P^2)$, where $ds_P^2$
is the Poincaré metric, Daniel obtained in \cite{Dan2}:
\begin{teo}[Daniel]\label{th:gauss0}
The Gauss map of a nowhere vertical minimal surface is harmonic into
$\H^2$.

Conversely, let $g:\Sigma\to\H^2$ be a harmonic map defined on  a
simply connected oriented Riemann surface into $\H^2$, and assume
that $g$ is nowhere antiholomorphic (i.e., $g_z$ does not vanish at
any point). Take $z_0\in\Sigma$ and $X_0\in\nil$.

Then there exists a unique conformal nowhere vertical minimal
immersion $X:\Sigma\to\Nil$ with $X(z_0)=X_0$ having $g$ as its
Gauss map. Moreover, $X$ can be uniquely recovered from $g$ through
an adequate representation formula.
\end{teo}

Furthermore, it can be checked that the Weierstrass data of $g$ as
above are $\{-Q,\lambda u^2\}$, where $Qdz^2$, $\lambda|dz|^2$ and
$u$ are, respectively, the Abresch-Rosenberg differential, the
metric, and the angle function of the surface defined in Section 2.

As we saw in Section 3, minimal surfaces in $\Nil$ and $H=1/2$
surfaces in $\hr$ are related by the sister correspondence, and
sister surfaces have the same metric and angle function (in
particular, the condition of being nowhere vertical is preserved).
As in this case the sister surfaces have opposite Abresch-Rosenberg
differentials, it turns out that their respective harmonic Gauss
maps are conjugate to each other.

The relation between minimal surfaces in $\nil$ and $H=1/2$ surfaces
in $\H^2\times \R$ can be made more explicit by means of the theory
of spacelike CMC surfaces in $\L^3$, as follows.

\begin{teo}[\cite{FeMi3}]\label{co:L3final}
Let $X=(F,t):\Sigma\to\Nil$ be a simply connected nowhere vertical
minimal surface with metric $\lambda|dz|^2$ and angle function $u$,
and $\psi=(N,h):\Sigma\to\hr$ its sister surface.

Then $f:=(F,h):\Sigma\to\L^3$ is a spacelike $H=1/2$ surface in the
Minkowski $3$-space with metric $\lambda{u^2}|dz|^2$ and Hopf
differential $-Qdz^2$, where $Qdz^2$ is the Abresch-Rosenberg
differential of $X$.
\end{teo}

\subsubsection{CMC $\sqrt{-\kappa}/2$ surfaces in $\widetilde\psl$}

In a forthcoming paper \cite{progress}, the authors and B. Daniel
will prove that there exists also a harmonic Gauss map for critical
CMC surfaces in the remaining case, i.e. the universal covering of
the group $\psl$, and will derive a representation formula for them.


\subsection{Half-space theorems}

One of the most important results in the global study of minimal
surfaces in $\R^3$ is the classical half-space theorem by Hoffman
and Meeks \cite{HoMe}. This theorem says that any properly immersed
minimal surface in $\R^3$ lying in a half-space must be a plane
parallel to the one determining the half-space. The main tools used
here are the maximum principle and the existence of catenoids, a
$1$-parameter family of minimal surfaces converging to a
doubly-covered punctured plane $P$, and intersecting the planes
parallel to $P$ in compact curves.

The analogous version for CMC one half surfaces in $\hr$ was proved
in \cite{HRS}. In this setting, horocylinders play the role of the
planes in $\R^3$.

\begin{teo}[Hauswirth-Rosenberg-Spruck]\label{th:halfspace12}
The only properly immersed CMC one half surfaces in $\hr$ that are
contained in the mean convex side of a horocylinder $C$ are the
horocylinders parallel to $C$.

Also, the only properly embedded CMC one half surfaces in $\hr$
containing a horocylinder in its mean convex side are the
horocylinders.
\end{teo}

\begin{proof}
The main point here is to construct a family of CMC one half
surfaces in $\hr$ to be used in the same way as catenoids in the
proof of the half-space theorem in $\R^3$. This is achieved by means
of compact annuli with boundaries, contained between two
horocylinders.
\end{proof}

For the case of $\Nil$, we must distinguish between horizontal and
vertical half-spaces. The equivalent to the half-space theorem for
surfaces lying in a horizontal half-space is proved by using the
family of rotational annuli \cite{AbRo2}. The corresponding vertical
version has been obtained in \cite{DaHa}, by constructing first a
family of \emph{horizontal catenoids}, i.e. properly embedded
minimal annuli (non-rotational) with a geometric behaviour good
enough to apply the Hoffman-Meeks technique.

\begin{teo}[Daniel-Hauswirth]\label{th:halfspace0}
The only properly immersed minimal surfaces in $\Nil$ that are
contained in a vertical half space are the vertical planes parallel
to the one determining the half-space.
\end{teo}

\begin{proof} Using the representation formula for minimal surfaces in $\Nil$ (see Theorem \ref{th:gauss0}),
it is possible to construct  {\em horizontal catenoids} in $\Nil$.
These surfaces are a $1$-parameter family of properly embedded
minimal annuli, intersecting vertical planes $\{x_2=c\}$ in a
non-empty closed convex curve. Moreover, the family converges to a
double covering of $\{x_2=0\}$ minus a point. They are obtained by
integrating a family of harmonic maps that belong to a more general
family used in the construction of Riemann type minimal surfaces in
$\hr$ \cite{Ha}. Once we have these \emph{catenoids}, we finish by
using the maximum principle similarly to the Euclidean case.
\end{proof}

\subsection{The classification of entire graphs}

In this section we will describe the space of entire graphs of
critical CMC in $\Ek$. Such a description follows from the works of
Fernández-Mira \cite{FeMi1,FeMi3}, Hauswirth-Rosenberg-Spruck
\cite{HRS} and Daniel-Hauswirth \cite{DaHa}, and is contained in
Theorems \ref{teonuevo1} and \ref{th:entire0}. We expose here a
unified perspective to this subject. First, we have

\begin{teo}[\cite{DaHa,FeMi3,HRS}]\label{teonuevo1}
The following conditions are equivalent for a surface of critical
CMC in $\Ek$:
 \begin{enumerate}
   \item[$(1)$]
   It is an entire graph.
\item[$(2)$]
It is a complete multigraph.
 \item[$(3)$]
$u^2 \, ds^2$ is a complete Riemannian metric (where $u$ is the
angle function and $ds^2$ the metric of the surface).
 \end{enumerate}
In particular, the sister correspondence preserves entire graphs of
critical CMC.
\end{teo}

Let us make some comments on this theorem. First, Hauswirth,
Rosenberg and Spruck proved $(2)\Rightarrow (1)$ for $H=1/2$
surfaces in $\H^2\times \R$. Second, the authors proved in
\cite{FeMi3} that $(3)\Rightarrow (1)$ (for any surface in $\Ek$,
not necessarily CMC), and that $(1)\Rightarrow (3)$  holds for
minimal surfaces in $\nil$. Finally, Daniel and Hauswirth showed
that $(2)\Rightarrow (1)$ holds for minimal surfaces in $\nil$. The
rest of the cases can be easily obtained from these results and the
sister correspondence (this was first observed in \cite{DHM}).

\begin{proof}
It is immediate that $(1)\Rightarrow (2)$. Also, by an eigenvalue
estimate, the authors proved in \cite{FeMi3} that for arbitrary
surfaces in $\Ek$ it holds $u^2 \, ds^2 \leq g_F$, where $F=\pi
\circ \psi$ is the projection onto $\Mk$ of $\psi$. Thus, if $u^2\,
ds^2$ is complete, $F$ is a local diffeomorphism with complete
pullback metric, and by standard topological arguments, $F$ is a
diffeomorphism, i.e. $(3)\Rightarrow (1)$ holds.

That $(1)\Rightarrow (3)$ holds for minimal surfaces in $\nil$ was
also proved in \cite{FeMi3}: let $X=(F,t):\Sigma\flecha \nil$ be an
entire minimal graph. By Theorem \ref{co:L3final}, there is an
entire spacelike CMC graph $f=(F,h):\Sigma\flecha \L^3$, whose
induced metric is $ds_f^2 = u^2 \, ds^2$. Now we can apply a theorem
by Cheng and Yau \cite{ChYa} which says that spacelike entire CMC
graphs in $\L^3$ have complete induced metric. Hence $u^2 \, ds^2$
is complete, as wished.

We will now prove that $(2)\Rightarrow (1)$ holds for minimal
surfaces in $\nil$. Let us observe that once this is done, we can
also prove the theorem for surfaces of critical CMC in all the
spaces $\Ek$. Indeed, as any simply connected surface of critical
CMC is the sister surface of some minimal surface in $\nil$, and as
the correspondence preserves the metric and the angle function
(therefore it preserves conditions $(2)$ and $(3)$ by passing to the
universal covering), we can easily translate the theorem for the
case of minimal surfaces in $\nil$ to the rest of the spaces. It is
important here that we proved $(3)\Rightarrow (1)$  in all spaces.

So, we only need to prove $(2)\Rightarrow (1)$ for minimal surfaces
in $\nil$. This was done by Daniel and Hauswirth \cite{DaHa}. For
that, they used their half-space theorem in $\nil$ (Theorem
\ref{th:halfspace0}) and an adaptation to $\nil$ of the previous
proof of $(2)\Rightarrow (1)$ for the case of $H=1/2$ surfaces in
$\H^2\times \R$ given by Hauswirth-Rosenberg-Spruck \cite{HRS}.

In order to prove $(2)\Rightarrow (1)$ for minimal surfaces in
$\nil$, we argue by contradiction. Assume that there exists a
complete multigraph $\Sigma$ that is not entire. Then there exists
an open set $\Sigma_0\subset\Sigma$ that is a graph over a disc
$D\subsetneq\R^2$ of a function $f$, and a point $q\in\partial D$
such that $f$ does not extend to $q$.

{\em Step 1: For any sequence of points $\{q_n\}$ in $D$ converging
to $q$, the sequence of normal vectors at the points
$p_n=(q_n,f(q_n))\in\Sigma_0$ converges to the horizontal vector
orthogonal to $\partial D$ at $q$.}

Indeed, as the surface is a multigraph, its angle function
$u=\esc{N}{E_3}$, where $N$ denotes the unit normal vector, is a
Jacobi function that does not vanish. As a result of this, $\Sigma$
is (strongly) stable, and has bounded geometry. This means that
locally around any $p_n$ we can write the surface as the graph (in
exponential coordinates) over a disc of radius $\delta$ of its
tangent plane, where $\delta$ is a universal constant depending only
on $\Sigma$. This neighborhood of $p_n$ will be denoted by
$\mathcal{G}(p_n)$. The limit of the normal vectors $\{N(p_n)\}$
must be a horizontal vector since otherwise, the piece
$\mathcal{G}(p_n)$ of bounded geometry could be extended as a graph
beyond $q$, which is impossible. Moreover, the limit vector must be
normal to $\partial D$ at $q$ since $\Sigma_0$ is a graph over $D$.

{\em Step 2: The function $f$ defining the graph $\Sigma_0$ diverges
at $q$. Moreover, as we approach $q$, and after translating the
surface to the origin, the surfaces converge to a piece of the
(translated) vertical plane $P$ passing through $q$ and tangent to
$\partial D$.}

That $f$ diverges at $q$ is a consequence of the completeness of
$\Sigma$, and the last part can be proved by following the ideas of
Collin and Rosenberg in \cite{CoRo}. We will assume that $P$ is the
plane $\{x_1=c\}$.

{\em Step 3: $\Sigma$ contains a graph $\mathcal{G}$ over a domain
of the form $U_\epsilon=(c-\epsilon,c)\times\R\subset\R^2$.
Moreover, this graph is disjoint from $P$ and asymptotic to it as
one approaches $q$.}

The graph $\mathcal{G}$ is obtained by analytical continuation of
the surfaces $\mathcal{G}(p_n)$ used in the first step, and after a
careful study of the behavior of the intersection curves of these
graphs and the planes parallel to $P$.

Finally, the contradiction follows from the half-space theorem
(Theorem \ref{th:halfspace0}). Recall that, although  $\mathcal{G}$
has   boundary and the theorem is formulated for surfaces without
boundary, its proof applies to this case, and so we are done.
\end{proof}

Once here, we investigate the \emph{Bernstein problem} for entire
graphs of critical CMC in $\Ek$, i.e. the classification of such
entire graphs (recall here that CMC graphs in $\Ek$ satisfy the
elliptic PDE \eqref{eq:graphs}). The terminology comes from the
classical Bernstein theorem: \emph{entire minimal graphs in $\R^3$
are planes}. Equivalently, any solution to the minimal graph
equation
 \begin{equation}\label{minclas}
   (1+f_y^2)f_{xx} -2f_x f_y f_{xy} + (1+f_x^2)f_{yy} = 0
 \end{equation}
defined on the whole plane is linear.

It is interesting to compare this result with the Bernstein problem
in $\nil$, i.e. the classification of entire minimal graphs in
$\nil$. This corresponds to classifying all global solutions to the
PDE \eqref{min}. Observe that taking $\tau=0$ in \eqref{min} we
obtain the classical equation \eqref{minclas}, i.e. the classical
case considered by Bernstein appears as a limit of the Heisenberg
case.

There exists, however, a great difference between both situations.
The following result classifies the entire graphs of critical CMC in
$\Ek$, by parametrizing the moduli space of such entire graphs in
terms of holomorphic quadratic differentials. It was obtained first
for minimal graphs in $\nil$ by the authors \cite{FeMi3}, and
shortly thereafter by Daniel and Hauswirth \cite{DaHa} for $H=1/2$
graphs in $\H^2\times \R$. The general case follows easily from the
Heisenberg case and Theorem \ref{teonuevo1}, using the sister
correspondence (this was observed first in \cite{DHM}).

\begin{teo}[Fern\'andez-Mira, Daniel-Hauswirth]\label{th:entire0}
Let $Q dz^2$ denote a holomorphic quadratic differential on
$\Sigma\equiv \C$ or $\D$, such that $Q\not\equiv 0$ if $\Sigma
\equiv \C$, and let $H^2 = -\kappa /4$.

There exists a $2$-parameter family of entire CMC $H$ graphs in
$\Ek$ whose Abresch-Rosenberg differential agrees with $Q dz^2$.
These graphs are generically non-congruent.

And conversely, these are all the entire graphs of critical CMC in
$\Ek$.
\end{teo}

At this point, the proof for the case of minimal surfaces in $\nil$
is a consequence of Theorem \ref{co:L3final} and the following
result by Wan and Wan-Au \cite{Wan,WaAu} on spacelike entire CMC
graphs in $\L^3$: \emph{for any holomorphic quadratic differential
as above, there exists a unique (up to isometries) spacelike entire
CMC $1/2$ graph in $\L^3$ with Hopf differential $Qdz^2$.} The
$2$-parameter family of non-congruent graphs in $\Ek$ comes from the
loss of ambient isometries (from $6$ dimensions to $4$ dimensions)
when passing from $\L^3$ to $\Nil$.

The remaining cases of critical CMC graphs follow since by Theorem
\ref{teonuevo1} the sister correspondence preserves entire graphs.

\subsection{Open Problems}
As explained in Section 3, entire graphs are stable. It is
conjectured that entire graphs and vertical cylinders are the only
stable critical CMC surfaces (this has been proved for parabolic
conformal type in \cite{MaPR}). Related to this is the question of
non-existence of complete stable $H>1/2$ surfaces in $\H^2\times \R$
(proved for $H>1/\sqrt{3}$ by Nelli-Rosenberg, \cite{NeRo2}).

Also, not much is known about properly embedded surfaces of critical
CMC and non-trivial topology. Can one obtain them by conjugate
Plateau constructions, or by integrable systems techniques? Another
remarkable problem is to establish the strong half-space theorem in
$\nil$: are two disjoint properly embedded minimal surfaces in
$\nil$ necessarily two parallel vertical planes, or two parallel
entire minimal graphs?


\section{Minimal surfaces in $\H^2\times \R$ and $\S^2\times \R$}

Minimal surfaces in product spaces admit a special treatment, due to
several reasons. One of them is the following: if
$\psi=(N,h):\Sigma\flecha M^2\times \R$ is a minimal surface
immersed in the product space $M^2\times \R$, where $(M^2,g)$ is a
Riemannian surface, then the horizontal projection $N:\Sigma\flecha
M^2$ is a harmonic map and the height function $h:\Sigma\flecha \R$
is a harmonic function. This implies, for instance, that compact
minimal surfaces in $M^2\times \R$ only exist if $M^2$ is compact
(in particular, if $M^2=\S^2$), and the only ones are the slices
$M^2\times \{t_0\}$.

Another important fact about minimal surfaces in $M^2\times \R$ is that there is a natural notion of \emph{minimal graph} over a domain $\Omega\subset M^2$, and that this graph satisfies a simple elliptic PDE in divergence form. This fact together with general existence results for solutions to the Plateau problem in Riemannian $3$-manifolds allows a good control on the geometry of the surface. Some of the most interesting results of the theory of minimal surfaces in product spaces come from the interplay between the information provided by harmonic maps and by Plateau constructions and the minimal graph equation.

Starting with the pioneer work of H. Rosenberg \cite{Ros}, and W.H
Meeks and H. Rosenberg \cite{MeRo1,MeRo2}, the theory of minimal
surfaces in $M^2\times \R$ has developed substantially in the last
decade. We will only talk here about a few results of special
relevance to the theory, and not mention many other important
results.

\subsection{The Collin-Rosenberg theorem}

The classical Bernstein theorem in $\R^3 $ states that planes are the only entire minimal graphs in $\R^3$. This theorem can be extended to the case of product spaces: \emph{any entire minimal graph in $M^2\times \R$, where $(M^2,g)$ is a complete surface of non-negative curvature, is totally geodesic}.

In contrast, in the product space $\H^2\times \R$ there is a wide variety of entire minimal graphs. For instance, in \cite{NeRo1} Nelli and Rosenberg solved the Dirichlet problem at infinity for the minimal graph equation in $\H^2\times \R$. They proved that any Jordan curve at the ideal boundary $\S^1\times \R \equiv \parc_{\8} \H^2\times \R$ of $\H^2\times \R$ which is a graph over $\S^1\equiv \parc_{\8} \H^2$ is the asymptotic boundary of a unique entire minimal graph in $\H^2\times \R$ (see \cite{GaRo} for a proof of this in the more general case of entire minimal graphs in $M^2\times \R$, where $(M^2,g)$ is complete, simply connected and with $K_M\leq c<0$).

All these entire minimal graphs are hyperbolic, that is, they have
the conformal type of the unit disk. The problem of existence of
entire minimal graphs of parabolic type (i.e. with the conformal
type of $\C$) is much harder, and was solved recently by Collin and
Rosenberg \cite{CoRo}.

\begin{teo}[Collin-Rosenberg]
There exist entire minimal graphs in $\H^2\times \R$ of parabolic
conformal type.
\end{teo}
As the projection onto $\H^2$ of a minimal graph is a harmonic
diffeomorphism, the above theorem has the following consequence,
which solves a major problem in the theory of harmonic maps and
disproves a conjecture by R. Schoen and S.T. Yau.
\begin{cor}[Collin-Rosenberg]
There exist harmonic diffeomorphisms from $\C$ onto $\H^2$.
\end{cor}

The proof by Collin and Rosenberg is a good example of the
interaction between the harmonicity properties of the minimal
immersion and the use of Plateau constructions and the minimal graph
equation.

The main idea in the proof is to construct first (non-entire) minimal graphs in $\H^2\times \R$ of Scherk type over ideal geodesic polygons, having alternating asymptotic values $+\8$ and $-\8$ on the sides of the polygon. This generalizes a classical construction by Jenkins and Serrin in the case of minimal graphs over bounded domains in $\R^3$. This construction is done as follows:

Let $\Gamma$ be an ideal polygon of $\H^2$, so that all the vertices
of $\Gamma$ are at the ideal boundary of $\H^2$ and $\Gamma$ has an
even number of sides $A_1,B_1,A_2, B_2...,A_k,B_k$, ordered
clockwise. At each vertex $a_i$, we consider a small enough
horocycle $H_i$ with $H_i \cap H_j = \emptyset$. Each $A_i$ (resp.
$B_i$) meets exactly two horocycles. Denote by $\tilde A_i$ (resp.
$\tilde B_i$), the compact arc of $A_i$(resp $B_i$) which is the
part of $A_i$ outside the two horodisks. We denote by $|A_i|$ the
length of $|\tilde A_i|$. Define $\tilde B_i$ and  $|B_i|$ in the
same way.

Now we can consider $a(\Gamma)= \sum _{i=1}^k |A_i|$ and $b(\Gamma)
= \sum _{i=1}^k |B_i|$. We observe that $a(\Gamma)-b(\Gamma)$ does
not depend on the choice of the horocycle $H_i$ at $a_i$, since
horocycles with the same point at infinity are equidistant. Keeping
in mind these data, we can state the following theorem by
Collin-Rosenberg \cite{CoRo} (see also Nelli-Rosenberg
\cite{NeRo1}):

\begin{teo}(\cite{NeRo1}, \cite{CoRo}) \label{js} There is a (unique up to additive constants) solution to the minimal surface equation in the polygonal domain
$P$, equal to $+\infty$ on $A_i$ and $-\infty$ on $B_i$, if and only
if the following conditions are satisfied:
\begin{enumerate}
\item $a(\Gamma)=b(\Gamma)$,
\item For each inscribed polygon ${\cal P}$ in $\Gamma$, ${\cal P} \neq \Gamma$, and for some choice of horocycles at the vertices, one has
$$2 a({\cal P}) < |{\cal P}| \hbox{ and } 2 b({\cal P}) < |{\cal P}|.$$
\end{enumerate}
\end{teo}

All these examples have the conformal type of $\C$. Once there,
Collin and Rosenberg designed a way of enlarging a given Scherk-type
graph over the interior of some $\Gamma\subset \H^2$ into another
one with more sides, and so that: (1) the extended surface is
$C^2$-close to the original one over an arbitrary compact set in the
interior of $\Gamma$, and (2) there is a control on the conformal
radius on adequate compact annuli on the surface.

By passing to the limit in this sequence of minimal graphs over
larger and larger domains, they obtained an entire minimal graph in
$\H^2\times \R$ which, by the control on the conformal radii of
these annuli, has the conformal type of $\C$.

\begin{rem}
The Collin-Rosenberg theorem has been extended by J.A. Gálvez and H. Rosenberg \cite{GaRo} to more general product spaces $M^2\times \R$: \emph{there exist entire minimal graphs of parabolic conformal type on $M^2\times \R$, where $(M^2,g)$ is any complete simply connected Riemannian surface with Gaussian curvature $K_M\leq c<0$ ($K_M$ not constant).}
\end{rem}

\subsection{Minimal surfaces of finite total curvature in
$\H^2\times \R$}

One of the most studied families among minimal surfaces in $\R^3$
are the complete minimal surfaces of finite total curvature (FTC for
short). A minimal surface $\Sigma$ is said to have FTC if its
Gaussian curvature $K$ satisfies
$$\left |\int_\Sigma K \,dA\right | <\infty.$$

By classical theorems of Huber and Osserman, complete FTC minimal
surfaces in $\R^3$ are conformally equivalent to a compact Riemann
surface minus a finite number of points. Moreover, the Gauss map
extends meromorphically to the punctures, and the total curvature of
the surface is a multiple of $-4\pi$. A key point here is that the
Gauss map of a minimal surface in $\R^3$ is conformal.

In $\hr$ there is no conformal Gauss map for minimal surfaces.
Nonetheless, using the global theory of harmonic maps into $\H^2$,
L. Hauswirth and H. Rosenberg \cite{HaRo} were able to prove that a
similar situation holds in $\H^2\times \R$.

\begin{teo}[Hauswirth-Rosenberg]
Let $X$ be a complete minimal immersion of $\Sigma$ in $\H^2 \times
\R$ with finite total curvature. Then
\begin{enumerate}
\item  $\Sigma$ is conformally equivalent to a Riemann surface punctured at a finite
number of  points, $\Sigma \equiv \overline{M}_g
-\{p_{1}....,p_{k}\}$.

\item  $Qdz^2:= h_z^2 \, dz^2$ is  holomorphic on $M$ and extends meromorphically to each
puncture. If we parameterize each puncture $p_{i}$ by the exterior
of a disk of radius $r$, and if $Q(z)dz^2=z^{2m_{i}}(dz)^2$ at
$p_{i}$, then $m_{i} \geqslant -1$.

\item The third coordinate $u$ of the unit normal tends to zero uniformly at each puncture.

\item The total curvature is a multiple of $2\pi$:
$$\int_\Sigma K d A=2\pi \left(2-2g-2k-\sum_{i=1}^k m_{i}\right).$$

\end{enumerate}

As a consequence, every end of a finite total curvature surface is
uniformly asymptotic to a Scherk type graph described in Theorem
\ref{js}.
\end{teo}

\begin{proof}
The first step is to prove that locally around an end, $Qdz^2$ only
has at most a finite number of zeroes. Then a Huber theorem and an
argument of Osserman give that the ends are conformally a punctured
disk, and $Qdz^2$ extends meromorphically to the puncture. The final
part of the behavior of $Qdz^2$ follows from the fact that $Qdz^2 =
h_z^2 dz^2$, where $h$ is the height function of the surface.

To prove that $u$ goes to $0$ at the ends, take an annular
neighborhood of an end where $Qdz^2$ does not vanish. Then
reparameterize this annulus by $w=\int\sqrt{Q}dz$. The metric
conformal factor in these coordinates satisfies a sinh-Gordon
equation, and the Gaussian curvature monotonically decreases to
zero. Then, estimates on the growth of solutions of the sinh-Gordon
equation allows one to conclude that, at a finite total curvature
end, the tangent plane becomes vertical and the metric becomes flat.

Finally, the expression for the total curvature follows from
Gauss-Bonnet formula and the estimates for the sinh-Gordon equation
obtained before.
\end{proof}

In \cite{HaRo}, the following question was also raised: \emph{are
there complete non simply connected minimal surfaces with FTC in
$\hr$?} Notice that rotational catenoids have infinite total
curvature. Actually, at that time, the only known complete FTC
minimal surfaces were the Scherk type graphs.

This question was positively answered by J. Pyo \cite{Pyo} and also,
independently, by Rodr\'{i}guez and Morabito \cite{RoMo}. Pyo
constructed a $1$-parameter family of genus zero properly embedded
minimal surfaces in $\hr$ with $k$ ends for $k \geq 2$, similar to
the $k$-noids in $\R^3$ (although the first ones are embedded and
the $k$-noids in $\R^3$ are not). They have total curvature
$4\pi(1-k)$, and are asymptotic to vertical planes at infinity.
These surfaces are obtained as the conjugate surfaces of minimal
graphs over infinite geodesic triangles in $\H^2$ that are
asymptotic to vertical planes at infinity.

Very shortly thereafter, M. Rodr\'{i}guez and F. Morabito discovered
independently a larger family of FTC minimal surfaces, containing
the previous ones. It is a $(2k-2)$-parameter of properly embedded
FTC minimal surfaces of genus zero with $k$ ends, obtained as the
limits of simply periodic minimal surfaces called {\em saddle
towers}, that  are invariant by a vertical translation of vector
$(0,0,2l)$. Taking limits when $l\to\infty$, they obtain genus zero
minimal surfaces with $k$ ends and total curvature $4\pi(1-k)$ that
are symmetric with respect to the reflection over the slice
$\H^2\times\{0\}$. The surfaces found by Pyo appear when the ends
are placed in symmetric positions.


\subsection{Open problems}

In \cite{Ha}, L. Hauswirth constructed a family of Riemann type
minimal surfaces in $\hr$ and $\S^2\times\R$, characterized by the
property of being foliated by curves of constant curvature. It is a
conjecture by W. Meeks and H. Rosenberg that in $\S^2\times\R$ they
are the only properly embedded minimal annuli. An approach for
solving this conjecture using integrable systems techniques has been
recently developed by L. Hauswirth and M. Schmidt. Another natural
problem is to obtain classification results for properly embedded
minimal surfaces of finite total curvature and a given simple
topology in $\R^3$.

Schoen and Yau proved there is no harmonic diffeomorphism from the
disk to a complete surface of non-negative curvature. Can there be
such a harmonic diffeomorphism onto a complete parabolic surface?
This is a question by J.A. Gálvez.

\def\refname{References}

\hspace{0.4cm}

\noindent The authors were partially supported by MEC-FEDER, Grant
No. MTM2007-65249, Junta de Andalucía Grant No. FQM325 and the
Programme in Support of Excellence Groups of Murcia, by Fundación
Séneca, R.A.S.T 2007-2010, reference 04540/GERM/06 and Junta de
Andalucía, reference P06-FQM-01642."


\begin{thebibliography}{9}

\bibitem[AbRo1]{AbRo1} U. Abresch, H. Rosenberg, A Hopf differential for constant mean curvature surfaces in $\S^2\times \R$ and $\H^2\times \R$, {\it Acta Math.} {\bf 193} (2004), 141--174.

\bibitem[AbRo2]{AbRo2} U. Abresch, H. Rosenberg, Generalized Hopf differentials, {\it Mat. Contemp.} {\bf 28} (2005), 1--28.

\bibitem[AEG1]{AEG1} J.A. Aledo, J.M.Espinar, J.A. Gálvez, Height estimates for surfaces with positive mean curvature in $\mathbb{M}\times \mathbb{R}$. {\em Illinois Journal of Math.,} {\bf 52}  (2008), 203-211.

\bibitem[Bon]{bon} F. Bonahon, Geometric structures on 3-manifolds. In Handbook of Geometric
Topology, pages 93--164. North-Holland, Amsterdam, 2002.

\bibitem[ChYa]{ChYa} S.Y. Cheng, S.T. Yau, Maximal spacelike hypersurfaces in the
Lorentz-Minkowski spaces, {\it Ann. of Math.} {\bf 104} (1976),
407--419.

\bibitem[CoRo]{CoRo} P. Collin, H. Rosenberg, Construction of harmonic diffeomorphisms and minimal graphs, {\it Ann. of Math.}, to appear (2007).

\bibitem[Dan1]{Dan1} B. Daniel, Isometric immersions into $3$-dimensional homogeneous manifolds, {\it Comment. Math. Helv.} {\bf 82} (2007), 87-131.

\bibitem[Dan2]{Dan2} B. Daniel, The Gauss map of minimal surfaces in the Heisenberg group, preprint, 2006,  arXiv:math/0606299.

\bibitem[DFM]{progress} B. Daniel, I. Fern\'andez, P. Mira, {Surfaces of critical
constant mean curvature}. Work in progress.

\bibitem[DaHa]{DaHa}  B. Daniel, L. Hauswirth, Half-space theorem, embedded minimal annuli and minimal graphs in the Heisenberg group. {\em Proc. Lond. Math. Soc. (3)}, {\bf 98} no.2 (2009), 445--470.

\bibitem[DHM]{DHM} B. Daniel, L. Hauswirth, P. Mira, Constant mean curvature
surfaces in homogeneous manifolds, preprint, 2009. Published
preliminarly by the Korea Institute for Advanced Study.

\bibitem[DaMi]{DaMi} B. Daniel, P. Mira, {Existence and uniqueness of constant mean curvature spheres in $\sol$.} Preprint, 2008, arXiv:0812.3059

\bibitem[dCF]{dCF} M.P. do Carmo, I. Fernández, Rotationally
invariant CMC disks in product space, {\it Forum Math.} {\bf 21}
(2009), 951--963.

\bibitem[EGR]{EGR} J.M. Espinar, J.A. Gálvez, H. Rosenberg, Complete surfaces with positive extrinsic curvature in product
spaces,  {\it Comment. Math. Helv.}, {\bf 84} (2009), 351--386.

\bibitem[EsRo]{EsRo2} J.M. Espinar, H. Rosenberg, Complete constant mean curvature surfaces
in homogeneous spaces, {\it Comment. Math. Helv.}, to appear (2009).

\bibitem[FeMi1]{FeMi1} I. Fern\'andez, P. Mira, Harmonic maps and constant mean curvature surfaces in $\H^2\times \R$,
{\it Amer. J. Math.} {\bf 129} (2007), 1145--1181.

\bibitem[FeMi2]{FeMi2} I. Fern\'andez, P. Mira, A characterization of constant mean curvature surfaces
in homogeneous $3$-manifolds, {\it Diff. Geom. Appl.}, {\bf 25}
(2007), 281--289.

\bibitem[FeMi3]{FeMi3} I. Fern\'andez, P. Mira, Holomorphic quadratic differentials and the Bernstein problem in Heisenberg space. {\it Trans. Amer. Math. Soc.}, {\bf 361}, no 11, (2009), 5737--5752.

\bibitem[FoWo]{wood} A.P. Fordy, J.C. Wood. Harmonic maps and integrable systems. Aspects of Mathematics, vol. E23, by Vieweg, Braunschweig/Wiesbaden, 1994.

\bibitem[GMM]{GMM} J.A. G\'alvez, A. Martínez, P. Mira, The Bonnet
problem for surfaces in homogeneous $3$-manifolds, {\it Comm. Anal.
Geom.} {\bf 16} (2008), 907--935.

\bibitem[GaRo]{GaRo} J.A. G\'alvez, H. Rosenberg, Minimal surfaces and harmonic diffeomorphisms from the complex plane onto a Hadamard surface. Preprint, 2008,  arXiv:0807.0997.

\bibitem[Ha]{Ha} L. Hauswirth, Minimal surfaces of Riemann type in three dimensional product manifolds. {\em Pacific J. Math.}, {\bf 224}, no.1 (2006), 91--117.

\bibitem[HaRo]{HaRo} L. Hauswrith, H. Rosenberg. Minimal surfaces of finite total curvature in $\H\times\R$. {\em Mat. Contemp.} {\bf 31} (2006), 65--80.

\bibitem[HRS]{HRS} L. Hauswirth, H. Rosenberg, J. Spruck. On complete mean curvature $\frac{1}{2}$ surfaces in $\hr$. {\em Comm. Anal. Geom.}, {\bf 16}, no.5 (2008), 989--1005.

\bibitem[HoMe]{HoMe} D. Hoffman, W. H. Meeks III. The strong halfspace theorem for minimal surfaces. {\em Invent. Math.} {\bf 101}, no.2 (1990), 373--377.

\bibitem[HsHs]{HsHs} W.Y. Hsiang, W.T. Hsiang, On the uniqueness of isoperimetric solutions and imbedded soap bubbles in noncompact symmetric spaces I, {\it Invent. Math.} {\bf 98} (1989), 39--58.

\bibitem[Lee]{lee} H. Lee. {Extension of the duality between minimal surfaces and maximal surfaces.} Preprint, 2009.

\bibitem[MaPR]{MaPR} M. Manzano, J. Pérez, M. Rodríguez. Parabolic stable surfaces with constant mean
curvature. Preprint, 2009, arXiv:0910.5373.

\bibitem[MPR]{stable} W. H. Meeks III, J. Pérez, A. Ros. {Stable constant mean curvature surfaces}, {\em Handbook of Geometric Analysis} no.1 (2008).

\bibitem[Mee]{Mee} W.H. Meeks. Constant mean curvature surfaces in
homogeneous $3$-manifolds. Preprint (2009).

\bibitem[MeRo1]{MeRo1} W.H. Meeks, H. Rosenberg, The theory of
minimal surfaces in $M\times \R$, {\it Comment. Math. Helv.} {\bf
80} (2005), 811--858.

\bibitem[MeRo2]{MeRo2} W.H. Meeks, H. Rosenberg, Stable minimal
surfaces in $M\times \R$, {\it J. Differential Geom.} {\bf 68}
(2004), 515--534.

\bibitem[MoUr]{MoUr} S. Montiel, F. Urbano, A Willmore functional
for compact surfaces in the complex projective plane, {\it J. Reine
Angew. Math.} {\bf 546} (2002), 139--154.

\bibitem[NeRo1]{NeRo1} B. Nelli, H. Rosenberg. Minimal surfaces in $\hr$. {\em Bull. Braz. Math. Soc.}, {\bf 33}, no.2 (2002), 263--292.

\bibitem[NeRo2]{NeRo2} B. Nelli, H. Rosenberg.  Global properties of constant mean curvature surfaces in $\hr$. {\em Pacific J. Math.} {\bf 226}, no-1 (2006), 137--152.


\bibitem[Oss]{oss} R. Osserman. A survey on minimal surfaces. Dover Publications Inc., New York, second edition, 1986.

\bibitem[Ped]{Ped} R. Pedrosa. The isoperimetric problem in
spherical cylinders, {\it Ann. Global Anal. Geom.} {\bf 26} (2004),
333--354.

\bibitem[Pyo]{Pyo} J. Pyo. New examples of minimal surfaces in $\hr$. Preprint, 2009, arXiv:0911.5577.

\bibitem[RoMo]{RoMo} M. Rodr\'{i}guez, F. Morabito. Saddle towers in $\hr$. Preprint, 2009, arXiv:0910.5676.

\bibitem[Ros]{Ros} H. Rosenberg, Minimal surfaces in $M^2\times \R$,
{\it Illinois J. Math.} {\bf 46} (2002), 1177--1195.

\bibitem[SaE]{SaE} R. Sa Earp, Parabolic and hyperbolic screw
motion surfaces in $\H^2\times\R$, {\it J. Austr. Math. Soc.}, {\bf
85} (2008), 113--143.

\bibitem[SaTo]{SaTo} R. Sa Earp, E. Toubiana, Screw motion surfaces in $\H^2\times \R$ and $\S^2\times \R$,
{\it Illinois J. Math.} {\bf 49} (2005), 1323--1362.

\bibitem[Sco]{scott} P. Scott. The geometries of 3-manifolds, {\it Bull. London Math. Soc.} {\bf 15} (1983), 401--487.

\bibitem[Tor]{Tor} F. Torralbo. Rotationally invariant constant mean curvature surfaces in homogeneous $3$-manifolds. {\it Diff. Geom. Appl.}, to appear (2009), arXiv:0911.5128.

\bibitem[ToUr]{ToUr} F. Torralbo, F. Urbano. Compact stable constant
mean curvature surfaces in the Berger $3$-spheres. Preprint, 2009, arXiv:0906.1439.

\bibitem[Wan]{Wan} T.Y. Wan, Constant mean curvature surface harmonic map and universal Teichmuller space, {\it J. Differential Geom.} {\bf 35} (1992), 643--657.

\bibitem[WaAu]{WaAu} T.Y. Wan, T.K. Au, Parabolic constant mean curvature spacelike surfaces, {\it Proc. Amer. Math. Soc.} {\bf 120} (1994), 559-564.

\end{thebibliography}
\end{document}